\documentclass[12pt]{article}
\usepackage{geometry}                
\usepackage{graphicx,color}
\usepackage{amssymb,amsmath,amsthm,mathrsfs}
\usepackage[all,cmtip]{xy}
\usepackage{epstopdf, comment}
\usepackage{esint}

\DeclareMathOperator*{\esssup}{ess\,sup}

\usepackage[pdftex,bookmarks,pdfnewwindow,plainpages=false,unicode,pdfencoding=auto]{hyperref}

 \hypersetup{
pdfauthor={Oleg Ivrii and Ilgiz Kayumov},
pdftitle={Makarov's principle for the Bloch unit ball},
pdfsubject={Geometric function theory},
bookmarksdepth={4}
}

\DeclareMathOperator{\Area}{Area}

\DeclareMathOperator{\supp}{supp }

\newtheorem{lemma}{Lemma}[section]
\newtheorem{theorem}{Theorem}[section]
\newtheorem{corollary}{Corollary}[section]

\theoremstyle{remark}
\newtheorem*{remark}{Remark}

\numberwithin{equation}{section}

\DeclareMathOperator{\im}{Im}
\DeclareMathOperator{\re}{Re}

\DeclareMathOperator{\var}{Var}

\DeclareMathOperator{\becker}{B}

\DeclareMathOperator{\hyp}{hyp}

\DeclareMathOperator{\zyg}{Z}
\DeclareMathOperator{\LIL}{LIL}

\DeclareMathOperator{\per}{per}

\newcommand{\dzy}{\,\frac{|dz|^2}{y}}
\newcommand{\rhoH}{\rho_{\mathbb{H}}}
\newcommand{\Hbar}{\overline{\mathbb H}}

\title{Makarov's principle for the Bloch unit ball}
\author{Oleg Ivrii and Ilgiz Kayumov}
\date{February 18, 2017}

\begin{document}

\maketitle

\begin{abstract}
Makarov's principle relates three characteristics of Bloch functions that resemble the variance
of a Gaussian: asymptotic variance, the constant in Makarov's law of iterated logarithm and the second derivative of the integral means spectrum at the origin.
While these quantities need not be equal in general, we show that the universal bounds agree if we take the supremum over the Bloch unit ball.
For the supremum (of either of these quantities), we give the estimate  $\Sigma^2_{\mathcal B} < \min(0.9, \Sigma^2)$, where $\Sigma^2$ is the analogous quantity associated to the unit ball in the
$L^\infty$ norm on the Bloch space.
This improves on the upper bound in Pommerenke's estimate
$0.685^2 < \Sigma^2_{\mathcal B} \le 1$.
\end{abstract}

\section{Introduction}
The Bloch space consists of analytic functions in the unit disk for which
\begin{equation*}
 \|b\|_{\mathcal B}:=\sup_{z \in \mathbb{D}} (1-|z|^2) |b'(z)| < \infty,
\end{equation*}
while a function $b_0$ belongs to the little Bloch space $\mathcal B_0$ if
$$\label{lbloch}
\lim_{|z|\to 1^-}(1-|z|^2)|b_0'(z)|=0.
$$
Makarov's principle \cite{ivrii-mak} is concerned with three characteristics of functions $b \in \mathcal B / \mathcal B_0$ that measure the growth of $b$ near the unit circle:
\begin{itemize}
\item The {\em asymptotic variance}
\begin{equation}
\label{eq:av}
\sigma^2(b) = \limsup_{r\to1} \frac{1}{2\pi |\log(1-r)|} \int_{|z|=r} |b(z)|^2 \, |dz|.
\end{equation}
\item The {\em LIL constant}
\begin{equation}
\label{eq:lil}
C_{\LIL}(b) =  \esssup_{\theta \in [0,2\pi)} \ \biggl\{
\limsup_{r \to 1} \frac{ |b(re^{i\theta})|}{\sqrt{\log \frac{1}{1-r} \log \log\log  \frac{1}{1-r}}} \biggr\}.
\end{equation}
\item The {\em integral means spectrum}
\begin{equation}
\beta_b(\tau) = \limsup_{r \to 1} \frac{1}{ |\log(1-r)|} \cdot \log \int_{|z|=r} \bigl |e^{\tau b(z)} \bigr | \, |dz|, \qquad \tau \in \mathbb{C}.
\end{equation}
\end{itemize}
The above quantities are unrelated for general Bloch functions, see \cite{BaMo, le-zinsmeister} for interesting examples. Nevertheless,
when one takes the supremum over natural classes of Bloch functions, the universal bounds coincide. In this paper, we prove Makarov's principle for the
Bloch unit ball:
\begin{theorem}
\label{main-thm}
$$
\Sigma^2_{\mathcal B}  \, := \, \sup_{\|b\|_{\mathcal B} \le 1} \sigma^2(b) \, = \, \sup_{\|b\|_{\mathcal B} \le 1} C_{\LIL}^2(b) \, = \,
\lim_{\tau \to 0} \frac{4}{|\tau|^2} \cdot \sup_{\|b\|_{\mathcal B} \le 1} \beta_b(\tau).
$$
\end{theorem}

We also give two different upper bounds for  $\Sigma^2_{\mathcal B}$. The first upper bound is an explicit estimate, while the second upper bound
is in terms of an analogous quantity associated to the unit ball of $\mathcal B/\mathcal B_0$ equipped with the ``$L^\infty$ norm.''

\begin{theorem}
\label{main-thm2}
$$
\Sigma^2_{\mathcal B} < \min(0.9, \Sigma^2).
$$
\end{theorem}

Previously, it was known that
 $C_{\LIL}(b) \le  \|b\|_{\mathcal B}$ which was first established by Pommerenke \cite{Pomm2} in 1985 who used an iterative scheme involving Hardy's identity.
Two other proofs of this fact appeared since then: Ba\~nuelos \cite{Ba} came up with a clever argument based on an $L^\infty$ estimate for the Littlewood-Paley $g_*$ function while Lyons \cite{lyons}  used hyperbolic Brownian motion.
We can also mention the efforts of Przytycki \cite{przytycki} who proved the weaker statement $C_{\LIL}(b) \le \frac{16}{\log 2} \, \|b\|_{\mathcal B}$ by emulating a method of Philipp and Stout for lacunary trigonometric series, as well as Makarov's original result \cite{makarov85} which says that $C_{\LIL}(b) \le C\|b\|_{\mathcal B}$ holds with some constant $C>0$.
Nevertheless, the question
whether ``$C_{\LIL}(b) \le  \|b\|_{\mathcal B}$'' was sharp remained open. The above theorem answers this question in the negative.

For the lower bound, Pommerenke  gave an example of a Bloch function with  $C_{\LIL}(b) \le  0.685\, \|b\|_{\mathcal B}$, see \cite[Theorem 8.10]{Pomm}.

\subsection[The L\string^infinity norm on the Bloch space]{The $L^\infty$ norm on the Bloch space}

The original statement of Makarov's principle from \cite{ivrii-mak} deals with another unit ball which is more natural from the point of view of quasiconformal geometry.
 There, $\mathcal B/\mathcal B_0$ is equipped with the ``$L^\infty$ norm'' coming from the representation $B = P(L^\infty)$ via the
 {\em Bergman projection}
\begin{equation}
\label{eq:bergman-def}
 P\mu(z)=\frac{1}{\pi} \int_{\mathbb{D}}\frac{\mu(w)}{(1-z\overline{w})^2} \, |dw|^2.
\end{equation}
That is,
\begin{equation*}
\| b \|_{\mathcal B/\mathcal B_0, \infty} := \inf_{P\mu \sim b} \| \mu \|_\infty,
\end{equation*}
 where the infimum is taken over all $\mu \in L^\infty(\mathbb{D})$ such that $P\mu = b + b_0$ with $b_0 \in \mathcal B_0$.
When one takes the supremum over the $(\mathcal B/\mathcal B_0, \infty)$ unit ball, one obtains a different constant:
\begin{equation}
\label{eq:sigma-def}
\Sigma^2 \, := \,  \sup_{\|\mu\|_\infty \le 1} \sigma^2(P\mu) \, = \, \sup_{\|\mu\|_\infty \le 1} C_{\LIL}^2(P\mu)
\, = \, \lim_{\tau \to 0} \frac{4}{|\tau|^2} \cdot \sup_{\|\mu\|_\infty \le 1} \beta_{P\mu}(\tau).
\end{equation}
The quantitity $\Sigma^2$ is naturally related to the problem of {\em dimensions of quasicircles}\/. If $D(k)$ is the maximal dimension of a $k$-quasicircle,
then according to \cite{AIPP, hedenmalm, qcdim},
$$
D(k) = 1 + \Sigma^2 k^2 + o(k^2), \qquad 0 < k < 1,
$$
$$
0.879 < \Sigma^2 < 1.
$$
The proof of Theorem \ref{main-thm} uses {\em fractal approximation techniques} and is quite similar to that of (\ref{eq:sigma-def}), but requires some modifications which we describe in this paper. In fact, the argument is applicable with any reasonable
norm on  $\mathcal B/\mathcal B_0$, for instance,
\begin{equation}
\label{eq:alternative-norms}
\| b \|_{\mathcal B/\mathcal B_0, m} = \limsup_{|z| \to 1} (1-|z|^2)^m |b^{(m)}(z)|, \qquad m \ge 2,
\end{equation}
lead to the constants $\Sigma^2_{\mathcal B, m}$.
The details will be given in Section \ref{sec:box}.

\subsection[Why Sigma\string^2(Bloch) < 1?]{Why $\Sigma^2_{\mathcal B} < 1$?}

We now describe the idea behind the bound ``$\Sigma^2_{\mathcal B} \le 0.9$.''
Suppose $b$ is a Bloch function of norm 1. Since $b$ is a holomorphic function, it is reasonable to expect that the {\em Bloch quotient} $|2b'/\rho| = |b'(z)|(1-|z|^2)$ is strictly less than 1 on average, where $\rho(z) = \frac{2}{1-|z|^2}$ is the density of the hyperbolic metric on the unit disk.
One way to make this precise is to say that
\begin{equation}
\label{eq:goal-bb}
 \alpha(R) \, := \,
  \sup_{\|b\|_{\mathcal B} \le 1} \biggl [ \, \sup_{B}  \fint_{B} \, \biggl |\frac{2b'}{\rho}(z) \biggr |^2 \, \rho^2|dz|^2 \, \biggr] \, < \, 1,
\end{equation}
where the inner supremum is taken over all balls  $B \subset \mathbb{D}$ of hyperbolic radius $R$.
Here, the notation $\fint \dots \rho^2|dz|^2$ suggests that we consider the average in the hyperbolic metric. We will later give a quantitative estimate for $\alpha(R)$, but in order to
prove (\ref{eq:goal-bb}), the following observation is sufficient:

\begin{lemma}
\label{lusin-privalov}
Suppose $b \in \mathcal B$ is a Bloch function with $\|b\|_{\mathcal {B}}=1$. There exists $R, S > 0$ such that
 any ball $B= B_{\hyp}(z, R) \subset \mathbb{D}$ of hyperbolic radius $R$ contains a ball $B_{\hyp}(\zeta, S)$ on which the Bloch quotient is less than $1/2$.
\end{lemma}

\begin{proof}
Since the Bloch quotient is invariant under automorphisms of the disk, it is enough to consider the case when $B$ is centered at the origin. 
If the lemma were false, then by the Lipschitz property of Bloch functions, there would exist a sequence of functions $b_n$ in the Bloch unit ball with $|2b_n'/\rho| > 1/10$ on $B_{\hyp}(0, n)$. A normal families argument would produce a Bloch function $b$ for which the Bloch quotient was strictly bounded away from 0 on the entire disk.
This would contradict the maximum modulus principle applied to $1/b'$.
\end{proof}

In Section \ref{sec:coefficients}, we will show that $\alpha(R)$ dominates the asymptotic variance and find a value of $R$ for which $\alpha(R) \le 0.9$.

\begin{lemma}
\label{alpha-lemma}
\begin{equation}
\label{eq:alpha-lemma}
\sup_{\|b\|_{\mathcal B} =1} \sigma^2(b) \, \le \, \inf_{R > 0} \alpha(R).
\end{equation}
\end{lemma}
Theorem \ref{main-thm} implies that the right hand side of (\ref{eq:alpha-lemma}) also bounds the  LIL constant and the quadratic behaviour of the integral means spectrum at the origin.

\subsection{An application to harmonic measure}

To conclude the introduction, we apply Makarov's principle for the Bloch unit ball to study metric properties of  harmonic measure  of
 simply-connected domains. Let $\mathbf{S}$ denote the collection
of conformal maps $f: \mathbb{D} \to \mathbb{C}$ satisfying $f(0) = 0$ and $f'(0) =1$. The {\em Becker class}\, $\mathbf{S}_{\becker} \subset \mathbf{S}$
consists of conformal maps for which $\| \log f'\|_{\mathcal B} \le 1$. According to {\em Becker's univalence criterion}\/, it is in  bijection with functions in the Bloch unit ball with $b(0) =0$.
A theorem of Makarov \cite{makarov87},
 \cite[Theorem VIII.2.1]{GM} shows:
\begin{corollary}
\label{main-cor2}
{\em (i)} Let $f \in \mathbf{S}_{\becker}$ be a function in the Becker class, $\Omega = f(\mathbb{D})$ be the image of the unit disk and $z_0$ be a point in $\Omega$.
The harmonic measure $\omega_{z_0}$ on $\partial \Omega$ as viewed from $z_0$ is absolutely continuous with respect to the Hausdorff measure  $\Lambda_{h(t)}$,
$$h(t)=t\,\exp\left\{C\sqrt{\log\frac{1}{t}\log\log
\log\frac{1}{t}}\right\},\qquad 0<t<10^{-7},$$
for any $C \ge C_{\LIL}(\log f')$. In particular, $C = \sqrt{\Sigma^2_{\mathcal B}}$ works.

{\em (ii)} Conversely, if $C < \sqrt{\Sigma^2_{\mathcal B}}$, there exists a conformal map in $\mathbf{S}_{\becker}$ for which $\omega_{z_0} \perp \Lambda_{h(t)}$.
\end{corollary}

The above corollary has a surprising consequence. As discussed in \cite{qcdim},
$$
\sup_{f \in \mathbf{S}} \sigma^2(\log f') \, \ge \, \Sigma^2 \, > \, \Sigma^2_{\mathcal B} \, 
= \, \sup_{f \in \mathbf{S}_{\becker}} \sigma^2(\log f'),
$$
which shows that functions in the Becker class are rather tame.
For more bounds on $\sup_{f \in \mathbf{S}} \sigma^2(\log f')$, we refer the reader to the works \cite{kayumov, HK, AIPP}.

\subsection*{Acknowledgements}

The first author was supported by the
Academy of Finland, project nos.~271983 and 273458. The second author was supported by the RFBR and the government of the Republic of Tatarstan, project nos.~14-01-00351 and 15-41-02433.



 \section{The Bergman projection}

\label{sec:background}

As mentioned in the introduction, the Bergman projection takes $L^\infty(\mathbb{D})$ to the Bloch space.
The estimate $\Sigma^2_{\mathcal B} < \Sigma^2$ is immediate
 from the following representation which is interesting in its own right:

\begin{lemma}
\label{bergman-representation}
 The Bergman projection $P: L^\infty(\mathbb{D}) \to \mathcal B/\mathcal B_0$
is surjective. Any $b \in \mathcal B$ may be represented
as
\begin{equation}
\label{eq:bergman-representation}
b = P\mu + b_0, \qquad \text{with} \quad \|\mu\|_\infty \le C \|b\|_{\mathcal B}, \qquad b_0 \in \mathcal B_0.
\end{equation}
Furthermore, the optimal constant $C$ in (\ref{eq:bergman-representation}) is strictly less than 1.
\end{lemma}

\begin{proof}
 For a Bloch function $b \in \mathcal B$, set
\begin{equation}
\label{eq:mub}
\mu_b(z) := (1/\overline{z}) \cdot 2b'/\rho.
\end{equation}
The reader may notice that $\mu_b$ is not bounded near the origin, however, only the asymptotic bound 
 $\limsup_{|z| \to 1} |\mu_b(z)| \le \| b \|_{\mathcal B}$ is essential here.
We claim that $P\mu_b(z) = b(z) - b(0)$. For this purpose, we consider
  the reproducing formula
 \begin{equation}
 \label{eq:reproducing-1}
 f(z) = \frac{2}{\pi} \int_{\mathbb{D}}  \frac{f(w)(1-|w|^2)}{(1-z\overline{w})^3} |dw|^2
 \end{equation}
 of the weighted Bergman space $A_1^2$ with norm
 \begin{equation}
\| f(z) \|_{A^2_1}^2 = \frac{2}{\pi} \int_{\mathbb{D}} |f(z)|^2 (1-|z|^2) |dz|^2,
 \end{equation}
for instance see \cite[Theorem 2.7]{zhu-spaces}. By considering dilates $f(rz)$ and taking $r \to 1^-$, it follows that (\ref{eq:reproducing-1}) holds 
for all holomorphic functions $f$ for which $$ \|f \|_{A^\infty_1} \, := \, \sup_{z \in \mathbb{D}} |f(z)| (1-|z|^2) \, < \, \infty.$$
This allows us to take $f = b'$ which gives $(P\mu_b)' = b'$, thus $P\mu_b$ and $b$ agree up to an additive constant. To evaluate this constant, note that $P\mu_b(0) = 0$ by the definition of the Bergman projection (\ref{eq:bergman-def}).
This proves the claim. 
 

To proceed further, we use duality considerations. Without loss of generality, we can assume that $\|b\|_{\mathcal B} = 1$ and $b(0) = 0$. As is well known, e.g.~see \cite[Theorem 3.17]{zhu-spaces} or \cite{hedenmalm}, the Bloch space equipped with the $L^\infty$ norm
$\| b \|_{\mathcal B, \infty} := \inf_{P\nu = b} \| \nu \|_\infty$ is the {\em isometric} dual of the Bergman space $A^1$ with respect to the pairing
$$\langle b, g \rangle =  \lim_{r \to 1} \frac{1}{\pi} \int_{\mathbb{D}} b(z) \overline{g(rz)} \, |dz|^2, \qquad
b \in \mathcal B, \ g \in A^1.$$
Since the Bergman projection is self-adjoint, the above duality implies
$$
\inf_{P\nu = b}  \, \|\nu\|_\infty \, = \, \sup_g \, \biggl | \int_{\mathbb{D}} {\mu}_b(z) \overline{g(z)} \, |dz|^2 \biggr |, \qquad \int_{\mathbb{D}} |g(z)| = 1.
$$
Readers familiar with  asymptotic Teichm\"uller spaces (see \cite[Chapter 14.10]{GL}) will recognize that the infimum of $\|\nu\|_\infty$ over
 $\nu \in L^\infty(\mathbb{D})$ satisfying $P\nu = b + b_0$ with $b_0 \in \mathcal B_0$ and $b_0(0) = 0$ is given by
\begin{equation}
\label{eq:degenerating-hamilton}
\inf_{P\nu = b + b_0} \, \|\nu\|_\infty \, = \, \sup_{\{g_n\}} \, \limsup_{n \to \infty} 
\, \biggl | \int_{\mathbb{D}} {\mu}_b(z) \overline{g_n(z)} \, |dz|^2 \biggr |,
\end{equation}
where the supremum is taken over all {\em degenerating} Hamilton sequences $\{g_n\} \subset A^1$, i.e.~sequences with $\int_{\mathbb{D}} |g_n| = 1$ for which $g_n \to 0$ uniformly on compact subsets of the disk. 
%
A normal families argument and Lemma \ref{lusin-privalov} show that a definite proportion of the
mass of $|g_n|$ is ``wasted'' on the set where $|2b'/\rho| < 1/2$.
Therefore, the right hand side of (\ref{eq:degenerating-hamilton}) is bounded by a constant strictly less than 1.
\end{proof}

\begin{remark}
One can compare Lemma \ref{bergman-representation} to a result of Per\"al\"a \cite{perala} which says that
the operator seminorm $\| P \|_{L^\infty \to \mathcal B} = 8/\pi$.
By contrast, the optimal constant $C$ in (\ref{eq:bergman-representation}) is not known.
\end{remark}

\newpage

\subsection{Working in the upper half-plane}

While it is simpler to visualize the duality arguments in the unit disk, the upper half-plane is a more natural setting for the
fractal approximation techniques.
We are therefore led to consider the Bloch space $\mathcal B(\mathbb{H})$ which consists of holomorphic functions on $\mathbb{H}$ with
\begin{equation}
\label{eq:def-bloch}
\|b\|_{\mathcal B(\mathbb{H})} = \sup_{z \in \mathbb{H}}\, 2y \cdot |b'(z)| < \infty.
\end{equation}
Instead of using the Bergman projection, we prefer to represent Bloch functions
in $\mathbb{H}$ via the {\em modified Beurling transform} \, $\mathcal S^\#: L^\infty(\Hbar) \to \mathcal B(\mathbb{H})$,
\begin{equation}
\label{eq:beurling22}
 \mathcal S^\# \mu(z) \, =\, -\frac{1}{\pi} \int_{\Hbar} \mu(w) \biggl [ \frac{1}{(w-z)^2} - \frac{1}{w(w-1)} \biggr ] \, |dw|^2,
 \end{equation}
 which includes the term $- \frac{1}{w(w-1)}$ to guarantee convergence. We often abuse notation and write
 $$ \text{``}(\mathcal S \mu)'(z) \text{''} \, := \, (\mathcal S^{\#} \mu)' (z)  \, =\, -\frac{2}{\pi} \int_{\Hbar} \frac{\mu(w)}{(w-z)^3}  \, |dw|^2.$$
 In this setting, (\ref{eq:mub}) becomes
 \begin{equation}
  \label{eq:mub2}
\mu_b =2i \cdot \overline{b'(\overline{z})}/\rho_{\Hbar}, \qquad  \mathcal S^\#\mu_b = b + C.
 \end{equation}

\subsection{Locality of the Beurling transform}

For our purposes, the most important property of the Beurling transform is that it is {\em local} in nature \cite[Section 4]{qcdim}:
\begin{lemma}
\label{qbounds} {\em (i)}
Suppose $\mu$ is a Beltrami coefficient supported on the lower half-plane with $\|\mu\|_\infty \le 1$. Then,
 $\bigl |(2(\mathcal S\mu)'/\rhoH)(z) \bigr | \le 8/\pi$.

{\em (ii)} For any $\varepsilon > 0$, there exists $R > 0$ sufficiently large so that if the hyperbolic distance $d_{\Hbar}(\overline{z}, \,\supp \mu) \ge R$ then
$\bigl |(2(\mathcal S\mu)'/\rhoH)(z) \bigr | \le \varepsilon.$
\end{lemma}

In particular, if $\mu_1$ and $\mu_2$ are two Beltrami coefficients on $\Hbar$ with $\| \mu_i\|_\infty \le 1$, $i=1,2$ that agree on
a ball $B_{\hyp}(\overline{z}, R) \subset \Hbar$ with $R$ large, then
\begin{equation}
\label{eq:locality6}
\biggl | \frac{ 2(\mathcal S\mu_1)'}{\rhoH}(z) - \frac{2( \mathcal S\mu_2)'}{\rhoH}(z) \biggr | \le 2 \varepsilon.
\end{equation}

\subsection{Boxes and grids}

By a {\em box} in the upper half-plane, we mean a rectangle
whose sides are parallel to the coordinate axes, with the bottom side located above the real axis.
Boxes naturally arise in grids. One natural collection of grids are the {\em $n$-adic grids} \,$\mathscr G_n$, defined for integer $n \ge 2$. An {\em $n$-adic interval}
$I \subset \mathbb{R}$
is an interval of the form $I_{j,k} = \bigl [j \cdot n^{-k}, (j+1) \cdot n^{-k} \bigr ]$. To an $n$-adic interval $I$, we associate
the {\em $n$-adic box} $$\square_I \, = \, \Bigl \{ w : \re w \in I, \, \im w \in \bigl [n^{-1}|I|, \,|I| \bigr ] \Bigr \}.
$$
It is easy to see that the boxes $\square_{I_{j,k}}$ with $j, k \in \mathbb{Z}$ have disjoint interiors and their union is $\mathbb{H}$.

If $\mu$ is a Beltrami coefficient supported on the lower half-plane,
we say that $\mu$ is {\em periodic} with respect to a grid $\mathscr G$ (or rather with respect to $\overline{\mathscr G}$) if for any two
boxes $\overline{B}_1, \overline{B}_2 \in \overline{\mathscr G}$, we have
$
\mu|_{\overline{B}_1} = L^*(\mu|_{\overline{B}_2})$, where $L(z) = az+b$, $a>0$, $b\in\mathbb{R}$, is the  affine map that takes $\overline{B}_1$ to $\overline{B}_2$. We remind the reader that when computing the pullback, Beltrami coefficients are viewed as $(-1,1)$-forms.

Given $\mu$ defined on a box $\overline{B}$, and a grid $\overline{\mathscr G}$ containing $\overline{B}$, there exists a unique periodic Beltrami coefficient $\mu_{\per}$
which agrees with $\mu$ on $\overline{B}$. As discussed in \cite{ivrii-mak}, for Bloch functions $b = \mathcal S^\# \mu$ with periodic $\mu$, the three characteristics in Makarov's principle are equal.

\subsection[An isoperimetric property]{An isoperimetric property of the metric $|dz|^2/y$}

The metric $|dz|^2/y$ will play an important role in the work. The crucial feature that makes the periodization arguments work is the following isoperimetric property: if $S > 0$ is held fixed, then
\begin{equation}
\label{eq:isoperimetry}
\frac{\Area(\partial_S \square, |dz|^2/y)}{\Area(\square, |dz|^2/y)} \to 0, \qquad \square \in \mathscr G_n, \quad n \to \infty,
\end{equation}
where $\partial_S \square := \{z \in \square : d_{\mathbb{H}}(z, \partial \square) < S \}$.

Combining the isoperimetric property of $|dz|^2/y$ with (\ref{eq:locality6}), we see that if $\mu_1 = \mu_2$ agree on the reflection $\overline{\square}$ of a box
$\square \in \mathcal G_n$, then the difference
 \begin{equation}
 \label{eq:boxcart5}
\Biggl | \,
 \fint_\square \, \biggl |\frac{2(\mathcal S\mu_1)'}{\rho_{\mathbb H}}(z) \biggr |^2 \dzy \, - \,
  \fint_\square \, \biggl |\frac{2(\mathcal S\mu_2)'}{\rho_{\mathbb H}}(z) \biggr |^2 \dzy \, \Biggr | \, = \, o(1), \quad \text{as }n \to \infty.
\end{equation}
Here, we have used the elementary identity $\bigl | |a|^2 - |b|^2 \bigr | \le |a-b| \cdot |a+b|$ and Lemma \ref{qbounds}(i).


\section{Box Lemma}
\label{sec:box}

For a Bloch function $b \in \mathcal B(\mathbb{H})$,
its {\em asymptotic variance} is given by
\begin{align}
\sigma^2_{[0,1]}(b) & = \limsup_{y \to 0^+} \, \frac{1}{|\log y|}  \int_0^1 |b(x+iy)|^2 \, dx, \\
\label{eq:mcm-ca}
 & = \limsup_{h \to 0^+} \, \frac{1}{|\log h|} \int_h^1 \int_0^1 \biggl |\frac{2b'(x+iy)}{\rhoH}\biggr |^2 \dzy.
 \end{align}
The equivalence of the two definitions is due to McMullen \cite[Section 6]{mcmullen}, for a purely upper half-plane proof, see
\cite{ivrii-mak}.
It is not difficult to show that the maximum asymptotic variance over the unit ball in $\mathcal B(\mathbb{H})$ coincides
 with the unit disk version:
$$\Sigma^2_{\mathcal B} = \sup_{\|b\|_{\mathcal B(\mathbb{H})} \le 1} \sigma^2_{[0,1]}(b).$$
The same is true for the other characteristics of Bloch functions from the introduction.
One can prove this by precomposing Bloch functions $b \in \mathcal B$ with the exponential $\xi(w) = e^{2\pi i w}$ like in \cite[Section 3]{qcdim}. We leave the details to the reader.

The proof of Makarov's principle for the Bloch unit ball rests on the following {\em Box Lemma}\,:
\begin{lemma}
\label{boxcart}
{\em (i)}
Given $\varepsilon > 0$, if $n \ge n(\varepsilon)$ is sufficiently large, then for any Bloch function $b$ with $\|b\|_{\mathcal B(\mathbb{H})} \le 1$ and $\square \in \mathscr G_n$,
 \begin{equation}
 \label{eq:boxcart}
 \fint_{\square} \, \biggl |\frac{2b'}{\rhoH}(z) \biggr |^2 \, \frac{|dz|^2}{y} \, < \, \Sigma^2_{\mathcal B} + \varepsilon.
 \end{equation}

 {\em (ii)} Conversely, for $n \ge n(\varepsilon)$ sufficiently large, there exists a Bloch function $b$ with $\|b\|_{\mathcal B(\mathbb{H})} \le 1$, which satisfies
 \begin{equation}
 \label{eq:boxcart-local-converse}
 \fint_{\square} \, \biggl |\frac{2b'}{\rho_{\mathbb{H}}}(z) \biggr |^2 \, \frac{|dz|^2}{y}  \, > \, \Sigma^2_{\mathcal B} - \varepsilon
 \end{equation}
 on every box $\square \in \mathscr G_{n}$. Furthermore,
 $b$ may be taken of the form $b = \mathcal S^\#\mu$
for some   Beltrami coefficient $\mu$ with $|\mu| \le \chi_{\Hbar}$ that is periodic with respect to the $n$-adic grid.
\end{lemma}

With the help of Lemma \ref{boxcart}, the proof of Theorem \ref{main-thm} can be completed using the scheme laid out in \cite{ivrii-mak}:

\begin{enumerate}
\item Following Makarov \cite{makarov90}, a Bloch function $b \in \mathcal B(\mathbb{H})$ defines an $n$-adic martingale on $[0,1]$. Namely, for an $n$-adic interval $I \subset [0,1]$, one can define
\begin{equation}
\label{eq:mak-martingale}
B_I := \lim_{h \to 0^+} \frac{1}{|I|} \int_{I+ih} b(x+ih) dx
\end{equation}
so that if $I_1, I_2, \dots, I_n$ are the $n$-adic children of $I$, then $B_I = (1/n) \sum_{j=1}^n B_{I_j}$.
\item Box averages (\ref{eq:boxcart}) describe the local variation of this martingale. Set
$$
\var_I B := \frac{1}{n}  \sum_{j=1}^n |B_{I_j} - B_I|^2.
$$
An  argument involving  Green's formula on $\square_I$ shows that
$$
\frac{\var_I B}{\log n} = \fint_{\square_I} \, \biggl |\frac{2b'}{\rho_{\mathbb{H}}}(z) \biggr |^2 \, \frac{|dz|^2}{y} + \mathcal O \bigl (\| b \|_{\mathcal B(\mathbb{H})}^2/\sqrt{\log n} \bigr ).
 $$
\item The local variance controls the three characteristics in Makarov's principle. Let $m = \inf_I \frac{\var_I B}{ \log n}$ and $M =\sup_I \frac{\var_I B}{\log n}$. One can show that the three characteristics
in Makarov's principle are pinched between $m$ and $M$. The key point is that the characteristics can be defined in terms of the martingale $B$, so the problem
is purely combinatorial.
\end{enumerate}

\begin{proof}[Proof of Lemma \ref{boxcart}]
(i) Assume for the sake of contradiction that there is a box $\square \subset \mathbb{H}$ and a function $b$  in the Bloch unit ball for which
 \begin{equation}
 \fint_{\square} \, \biggl |\frac{2b'}{\rhoH}(z) \biggr |^2 \dzy \, > \, \Sigma^2_{\mathcal B} + \varepsilon.
 \end{equation}
We can choose a Beltrami coefficient $\mu$ with $\mu \le \chi_{\Hbar}$ and $\mathcal S^\#\mu = b + C$, for instance
$\mu =2i \cdot \overline{b'(\overline{z})}/\rho_{\Hbar}$ will do, cf.~(\ref{eq:mub2}).
 We then form the Beltrami coefficient $\mu_{\per} \le \chi_{\Hbar}$ by restricting $\mu$ to $\overline{\square}$
 and periodizing with respect to $\overline{\mathscr G_n}$, that is, on $\overline{\square}_j \in \overline{\mathscr G_n}$, we define $\mu_{\per} = L_j^*\mu$, where $L_j(z) = az + b$
 is the unique affine mapping with $a > 0$ and $b \in \mathbb{R}$ that maps  $\overline{\square}_j$ to  $\overline{\square}$.
According to (\ref{eq:boxcart5}), we would have
 \begin{equation}
 \label{eq:boxcart7}
 \fint_{\square_j} \, \biggl |\frac{2(\mathcal S\mu_{\per})'}{\rhoH}(z) \biggr |^2 \dzy \, > \,
 \Sigma^2_{\mathcal B} + 2\varepsilon/3, \quad  \text{for all }\, \square_j \in \mathscr G_n,
  \end{equation}
  when $n$ is large.
However, this is not yet a contradiction since $\| \mathcal S^\#\mu_{\per} \|_{\mathcal B(\mathbb{H})}$ might be greater than 1.

For $z \in \overline{\square}_j$, let $h(z) = d_{\Hbar}(z, \partial \overline{\square}_j)$.
Fix a number $S > 0$ very large, but much smaller than $R = \log n$. We now modify $\mu_{\per}$ on the set
$W = \{z :h(z) < S\}.$ More precisely, we define a new Beltrami coefficient
\begin{equation}
\hat{\mu}_{\per}(z) \, = \, \left\{
\begin{array}{lr}
(h(z)/S) \cdot \mu_{\per}, & z \in W, \\
\mu_{\per}, & z \notin W.
\end{array} \right.
\end{equation}
By making $S$ large, we can make $\|\mathcal S^\#\hat{\mu}_{\per}\|_{\mathcal B(\mathbb{H})} \le 1 + \varepsilon/3$. Furthermore, if $R >\!\!> S$ is large,
 Lemma \ref{qbounds} and the isoperimetric property (\ref{eq:isoperimetry})  guarantee that
$$
\Bigl | \sigma^2_{[0,1]}(\mathcal S^\#\hat{\mu}_{\per}) - \sigma^2_{[0,1]}(\mathcal S^\#{\mu}_{\per})  \Bigr | \le \varepsilon/3.
$$
This contradicts the definition of $\Sigma_{\mathcal B}^2$ (we may divide ${\hat \mu}_{\per}$ by $1+\varepsilon/3$), so our initial assumption must have been wrong.

(ii)
Conversely, suppose $b$ is a Bloch function with $$\|b\|_{\mathcal B(\mathbb{H})} \le 1
\quad \text{and} \quad \sigma^2_{[0,1]}(b) \ge \Sigma^2 - \varepsilon/2.$$
Consider the $n$-adic grid $\mathscr G_{n}$.
By the pigeon-hole principle, there exists an $n$-adic box $\square$
for which the integral in (\ref{eq:boxcart}) is at least $\Sigma^2 - \varepsilon/2$.
Set $\nu = 2i \cdot \overline{b'(\overline{z})}/\rho_{\Hbar}$ so that $\mathcal S^\# \nu = b+C$.
Restricting $\nu$ to $\overline{\square}$ and periodizing over $n$-adic boxes produces a Beltrami coefficient $\hat \nu_{\per}$ which
satisfies
\begin{equation}
\label{eq:std-box}
 \fint_{\square_j} \, \biggl |\frac{2(\mathcal S{\hat \nu}_{\per})'}{\rho_{\mathbb{H}}}(z) \biggr |^2 \, \frac{|dz|^2}{y}  \, > \, \Sigma^2 - \varepsilon,
 \quad \text{for all }\, \square_j \in \mathscr G_{n}.
\end{equation}
This completes the proof.
\end{proof}

\begin{remark}
To see that the periodization argument works for the norms
$$
\| b \|_{\mathcal B/\mathcal B_0, m} = \limsup_{|z| \to 1} (1-|z|^2)^m |b^{(m)}(z)|, \qquad m \ge 2,
$$
it suffices to show that $\mu \to \bigl |(\mathcal S\mu)^{(m+1)}/\rhoH^m)(z) \bigr |$ satisfies a variant of Lemma \ref{qbounds}.
We leave it to the reader to generalize the argument from  \cite[Section 4]{qcdim}. With a bit more work, one can show that the above argument is applicable
to the Zygmund norm on the Bloch space. Here, one takes a Bloch function $b(z)$, forms $$f(z) = \int_0^z b(w)dw$$ and then restricts to the unit circle, cf.~\cite[Theorems II.3.4 and VII.1.3]{GM}.
A similar construction
involving fractional integration endows the Bloch space with the $C^\alpha$ norms ($0 < \alpha < 1$).
More precisely, if $b \in \mathcal B$ then $$f(z) = \int_{\mathbb{D}} \frac{b(w)}{(1-z\overline{w})^{2-\alpha}} |dw|^2$$ satisfies $\sup_{z \in \mathbb{D}} (1-|z|^2)^{1-\alpha} |f'(z)| \le C_\alpha \|b\|_{\mathcal B}$ and hence extends to a $C^\alpha$ function on the unit circle. In fact, if $b = P\mu$ with $\mu \in L^\infty(\mathbb{D})$ then Fubini's theorem shows that  $f(z) = \int_{\mathbb{D}} \frac{\mu(w)}{(1-z\overline{w})^{2-\alpha}} |dw|^2$ from which the growth bound is clear. Further, the value of the $C^\alpha$ quotient $\frac{|f(z_1)-f(z_2)|}{|z_1-z_2|^\alpha}$, $z_1, z_2 \in \mathbb{S}^1$ is determined (up to small error) by the values of $\mu$ in a neighbourhood of the Euclidean midpoint of the hyperbolic geodesic joining $z_1$ and $z_2$.
The reader interested in working out the details can examine \cite[Chapter 7]{zhu-spaces} and
\cite[Theorem II.3.2]{GM}. 

These constructions lead to the constants $\Sigma^2_{\mathcal B, m}$, $\Sigma^2_{\zyg}$ and $\Sigma^2_{C^\alpha}$.
We note however that one can construct (equivalent) norms on $\mathcal B/\mathcal B_0$ which cannot be described as  $\limsup_{|z| \to 1}$ of local behaviour.
 For these irregular norms, the periodization scheme does not work.
\end{remark}

\newpage

\section{Coefficient estimates}
\label{sec:coefficients}

In this section, we prove Lemma \ref{alpha-lemma} and give an explicit estimate for the quantity
$$
\alpha(R)  = \sup_{\|b\|_{\mathcal B(\mathbb{H})}  \le 1} \biggl\{ \fint_{B_{\hyp}(i,R)} \, \biggl |\frac{2b'}{\rho_{\mathbb{H}}}(z) \biggr |^2 \, \rho_{\mathbb{H}}^2|dz|^2 \biggr\}.
$$
 The  proof of Lemma \ref{alpha-lemma} uses the motif of {\em geometry of averages} which says that averages of high concentration control averages of low concentration, e.g.~the $L^\infty$ norm controls the $L^1$ norm of a function.

\begin{proof}[Proof of Lemma \ref{alpha-lemma}]  
Suppose $b \in \mathcal B(\mathbb{H})$ is a Bloch function of norm 1.
If we average
$$
 \fint_{B_{\hyp}(z,R)} \, \biggl |\frac{2b'}{\rho_{\mathbb{H}}}(z) \biggr |^2 \, \rho_{\mathbb{H}}^2|dz|^2  \le  \alpha(R)
 $$
over $z \in [0,1] \times [h, 1]$ with respect to the measure $|dz|^2/y$, we see that the quantity (\ref{eq:mcm-ca}) is bounded by $\alpha(R) + \mathcal O_R(1/|\log h|)$ as $h \to 0^+$.
Thus, $\sigma^2_{[0,1]}(b) \le  \alpha(R)$ as desired.
\end{proof}

We proceed to give a quantitative bound for $\alpha(R)$.
For this purpose, we switch back to the unit disk and consider a Bloch function $b \in \mathcal B$ of norm 1. Let us expand $b'$ as a power series:
 $$
 b'(z)=q_0+q_1 z +q_2 z^2 + q_3 z^3 + \cdots.
 $$
 To give a bound for
 $$
 \int_{|z| \leq r} |b'(z)|^2\, |dz|^2 =2\pi \sum_{k=0}^\infty \frac{r^{2k+2}}{2k+2}|q_k|^2,
 $$
 we need to estimate the coefficients $\{q_k\}$. Since $\|b\|_{\mathcal B} \le 1$, Cauchy's estimates show:
\begin{equation}
\label{eq:cauchy}
|q_k| \leq \frac{k+2}{2} \left(\frac{k+2}{k}\right)^{k/2}, \quad k \ge 1.
\end{equation}

We estimate the first few terms together, rather than one at a time. To do this, we will use Parseval's identity and the following general
principle (see  \cite{avhkay} or \cite[Satz 3.2.1]{bonk}):
{\em To maximize a continuous and increasing functional of $|q_0|$ and $|q_1|$ in the unit ball of $\mathcal B$, it is enough to consider the special functions
$b(z)=(3/4)\sqrt{3} \, S_a(z)^2$ where
$$
S_a(z)=\frac{z+a}{1+az}, \quad a \in [0,1/\sqrt{3}].
$$
} 
For these special functions,
 \begin{equation} \label{bb2}
q_0= (3/2)\sqrt{3} a(1-a^2), \quad
q_1=(3/2)\sqrt{3}(1-a^2)(1-3 a^2).
 \end{equation}
 We remark that the parameter $a$ in our paper is denoted by $z_0$ in \cite{bonk}, also Bonk works with the Bloch function rather than its derivative, which explains the discrepancy in the factor of 2 between our $q_1$ and Bonk's $a_2$.
 
With this principle in mind, we give two estimates, one when $|q_2|$ is small, and another when $|q_2|$ is large:

\begin{lemma} \label{b1} \it  If $|q_2| \leq 2$, then  for $r \leq 0.4$, the following inequality holds:
 \begin{equation} \label{bb1}
 \frac{r^2}{2}|q_0|^2+\frac{r^4}{4}|q_1|^2+\frac{r^6}{6}|q_2|^2 \leq \frac{r^2}{2}+\frac{2}{3} r^6.
 \end{equation}
\end{lemma}

\begin{proof}
Straightforward computations show that for all the special functions $S_a$,
$$
\frac{r^2}{2}|q_0|^2+\frac{r^4}{4}|q_1|^2 \leq \frac{r^2}{2}.
$$
Hence, the above estimate is true for any function in the Bloch unit ball. Together with the assumption $|q_2| \leq 2$, this gives (\ref{bb1}).
 \end{proof}

 \begin{lemma} \label{b2}  \it  If $|q_2| \ge 2$, then  for $r \leq 0.4$, we have
 \begin{equation} \label{bb3}
 \frac{r^2}{2}|q_0|^2+\frac{r^4}{4}|q_1|^2+\frac{r^6}{6}|q_2|^2+\frac{r^8}{8}|q_3|^2 \, \leq \, \frac{r^2}{2}+\frac{17}{24} r^6 +2.77 \, r^8.
 \end{equation}
 \end{lemma}

  \begin{proof}
   By Parseval's formula,
 $$
 |q_0|^2+|q_1|^2 s^2+|q_2|^2 s^4+|q_3|^2 s^6 + \cdots \, \leq \, \frac{1}{(1-s^2)^2}, \quad s \in[0,1].
 $$
 In particular,
 $$
 |q_3|^2 s^6 \leq \frac{1}{(1-s^2)^2} - 4 s^4
 $$
 since
  $|q_2| \ge 2$ by assumption. The choice $s^2=0.58$ gives
 $
 |q_3|^2 \leq 22.16
 $ or
 $$
 \frac{r^8}{8}|q_3|^2 \leq 2.77 \, r^8.
 $$
It remains to show that
 \begin{equation} \label{bb4}
 \frac{r^2}{2}|q_0|^2+\frac{r^4}{4}|q_1|^2+\frac{r^6}{6}|q_2|^2 \leq \frac{r^2}{2}+\frac{17}{24} r^6.
 \end{equation}
For this purpose, consider the function
 $$
 \phi(z)=\frac{1- \tau}{2} \Bigl ( b'(\sqrt{\tau} z)+b'(-\sqrt{\tau} z) \Bigr ) \, = \, (1-\tau)q_0+(1-\tau) \tau \, q_2 z^2+\cdots,
 $$
 where $\tau$ is an auxiliary parameter in $[0,1]$.
 By construction, $|\phi(z)| \leq 1$ in the unit disk. Applying the Schwarz lemma to $\phi(\sqrt{z})$ gives
 $$
 (1-\tau) \tau \, |q_2| \leq 1-(1-\tau)^2|q_0|^2.
 $$
 Maximizing over $\tau \in [0,1]$, we obtain
 $$
|q_2| \leq 2+2\sqrt{1-|q_0|^2}.
$$
 Therefore to prove (\ref{bb4}), we have to show that the inequality
 $$
\frac{r^2}{2}|q_0|^2+\frac{r^4}{4}|q_1|^2+ \frac{r^6}{6}\left(2+2\sqrt{1-|q_0|^2}\right)^2 \leq \frac{r^2}{2}+\frac{17}{24} r^6, \quad r \leq 0.4,
$$
is valid for all $q_0, q_1$ of the form (\ref{bb2}), with $a \in [0,1)$. Routine computations show that this is indeed the case.
We leave the details to the reader.
\end{proof}

\begin{lemma} \label{b3} If $r=0.4$ then
$$
 \frac{(1-r^2)}{\pi r^2} \int_{|z| \leq r} |b'(z)|^2 \, |dz|^2 \leq 0.8998.
 $$
 \end{lemma}

\begin{proof}
If $|q_2| \leq 2$, then from Lemma \ref{b1} and Cauchy's estimates (\ref{eq:cauchy}),
\begin{equation*} 
\sum_{k=0}^\infty \frac{r^{2k+2}}{2k+2}|q_k|^2 \leq   \frac{r^2}{2}+\frac{2}{3} r^6+ \sum_{k=3}^\infty \frac{r^{2k+2}}{2k+2} \left(\frac{k+2}{2}\right)^2 \left(\frac{k+2}{k}\right)^{k},
\end{equation*}
while if $|q_2| \ge 2$, Lemma \ref{b2} gives
\begin{equation*} 
\sum_{k=0}^\infty \frac{r^{2k+2}}{2k+2}|q_k|^2 \leq   \frac{r^2}{2}+\frac{17}{24} r^6 +2.77\, r^8+ \sum_{k=4}^\infty \frac{r^{2k+2}}{2k+2} \left(\frac{k+2}{2}\right)^2 \left(\frac{k+2}{k}\right)^{k}.
\end{equation*}
Plugging in the value $r = 0.4$, numerical computations show that in either case, the left hand side does note exceed 0.8998.
\end{proof}

Since the hyperbolic area of the ball $B(0,r)$ is
$$
\int_{|z| \le r} \rho^2 \,|dz|^2 = 2\pi \int_0^r \frac{4s}{(1-s^2)^2} \, ds = 4\pi \, \frac{r^2}{1-r^2},
$$
we see that Lemma  \ref{b3} is equivalent to the statement $$\alpha(\eta(0.4)) < 0.8998,$$ where $\eta(0.4)$ is the hyperbolic radius of the ball $\{z : |z| \le 0.4 \}$.
This shows $\Sigma^2_{\mathcal B} < 0.9$, which completes the proof of Theorem \ref{main-thm2}.


\bibliographystyle{amsplain}

\end{document}